\documentclass{llncs}

\usepackage{amsmath}
\usepackage{amssymb}
\usepackage{amsfonts}
\usepackage{stmaryrd}
\usepackage{enumerate}
\usepackage{mathrsfs}

\renewcommand\texttt[1]{\text{\tt #1}}

\def\contexts{\mathscr C}
\newcommand{\N}[0]{\mathbb{N}}
\newcommand{\Prop}[0]{U_0}
\newcommand{\Nat}[0]{\texttt{Nat}}
\newcommand{\Natw}[0]{\texttt{\bf{N}}}
\newcommand{\supt}[0]{\texttt{sup}\:\:}
\newcommand{\Set}[0]{\texttt{Set}\:\:}
\newcommand{\SET}[0]{\texttt{SET}}

\newcommand{\EXTn}[0]{\texttt{EXT}_n\:\:}
\newcommand{\INn}[0]{\texttt{IN}_n\:\:}
\newcommand{\Zero}[0]{\texttt{Zero}}
\newcommand{\lsq}[0]{\llbracket}
\newcommand{\rsq}[0]{\rrbracket}

\newcommand{\tomega}[0]{{\tilde\omega}}
\newcommand{\ZFC}{{\bf ZFC}}
\newcommand{\CIC}{{\bf CIC}}
\newcommand{\PAt}{{{\bf PA}^2}}
\newcommand{\pr}{^\prime}
\newcommand{\vepsilon}{\varepsilon}

\newcommand{\ie}{{\it i.e. }}

\def\teoaname{Theorem~\ref{MAIN}$\Rightarrow$}
\def\teobname{Theorem~\ref{MAIN}$\Leftarrow$}
\spnewtheorem{fact}{Fact}{\bf}{\itshape}
\spnewtheorem*{teoa}{\teoaname}{\bf}{\itshape}
\spnewtheorem*{teob}{\teobname}{\bf}{\itshape}
\spnewtheorem{myclaim}{Claim}{\bf}{\itshape}
\spnewtheorem*{lem1}{Lemma~\ref{Z2Lsimpcompl}}{\bf}{\itshape}

\title{Reflections on Trusting Lean}
\date{\today}

\author{Lorenzo Bresolin\and Marcello Mamino}
\institute{Universit\`a di Pisa}

\begin{document}

    \maketitle
    \begin{abstract}
      It is known, by works of Werner and Carneiro, that the type theory of
Lean (and indeed of $\CIC$ plus a version of choice) is equiconsistent with
$\ZFC_\infty$, that is, $\ZFC\,+\,$``there are infinitely many strongly
inaccessible cardinals''. We prove a finer connection between these two
formal systems; namely, given any formula $\varphi$ in the language of
second order arithmetic, $\varphi$ is provable in $\ZFC_\infty$ if and only if
{\it the same} formula, when viewed as a type in Lean, is inhabited.

    \end{abstract}

\section*{Introduction}

Our work can be illustrated by a thought experiment.
We will be dealing with the Lean~4 proof assistant (see~\cite{MoUl21} for
context) but the basic concern and methodology are more general.
Now, imagine a mathematician were presented with the construction, in Lean, of an inhabitant $p$ of the propositional type $$\begin{aligned}&\forall a:\Nat,\exists b:\Nat,\forall c:\Nat,\forall d:\Nat,\\&\quad\neg((c+2)*(d+2)=a+b)\:\:\land\\&\quad\neg((c+2)*(d+2)=a+b+2).\end{aligned}$$
Further imagine that our mathematician subscribes to the axioms of $\ZFC$
and that they are willing to accept the existence of infinitely many
strongly inaccessible cardinals.  Because of Werner's and Carneiro's
work~\cite{Wer97,Car19},
they accept that Lean is free of contradiction.  Should they also be
satisfied that the twin prime conjecture must be true? We claim yes.

Similarly to other popular proof assistants such as Rocq (formerly Coq),
Lean is based on the Calculus of Inductive Constructions~($\CIC$), see, for instance, \cite{CoHu88} and~\cite{CoPa90}.
We call $\CIC^+$ the variant of~$\CIC$ implemented by Lean~4.
This calculus
extends $\CIC$ by the addition of a form of the axiom of choice, propositional
extensionality, and proof irrelevance.
It has been established by Werner~\cite{Wer97} that $\CIC$ enriched with
axioms for choice, and hence~$\CIC^+$, proves the consistency of
$\ZFC$ extended with the existence of infinitely many strongly inaccessible cardinals, which
we will denote $\ZFC_\infty$. Vice-versa, Carneiro constructs in~\cite{Car19} a
set-theoretic model of $\CIC^+$ assuming $\ZFC_\infty$. Thus $\CIC^+$ and
$\ZFC_\infty$ are equiconsistent.
However mere equiconsistency between two formal theories is in no way
assurance enough that they prove the same theorems. For an example, $\ZFC$
and $\ZFC+\neg\text{Con}(\ZFC)$ are equiconsistent, but obviously do not
prove the same formulas -- and indeed, proving a statement in the latter
would hardly be persuasive of that statement's veracity.  We propose,
here, a more fine grained result: namely that $\CIC^+$ and~$\ZFC_\infty$
prove the same theorems in the language of second order arithmetic.

There is a further subtlety. Since the languages of $\CIC^+$
and~$\ZFC_\infty$ differ, we should be precise about the meaning of
proving {\it the same theorems}. 
To carry out our comparison, it is necessary to fix translations from second order arithmetic
to set theory and types respectively. We feel that this is one important
point that we raise, namely that there are {\it conventional} or {\it
natural} translations: those that the users of the respective logical systems
perceive as such. It is important, in our view, to prove the equivalence
of, say, $\CIC^+$  and~$\ZFC_\infty$, under the premise that mathematical
statement are encoded according to the conventional translations.
Given a sentence $\varphi$ in the language of second order arithmetic,
we define translations $\varphi_Z$ and~$\varphi_C$ of~$\varphi$ in the
languages of $\ZFC_\infty$ and~$\CIC^+$ respectively. For $\ZFC_\infty$,
the obvious choice is to map the natural numbers to the finite von Neumann
ordinals. For~$\CIC^+$, using to the finite Von Neumann ordinals
constructed inside Werner's model would be artificial; on the contrary,
the conventional choice is to map the natural numbers to an
appropriate inductive type~$\Nat$,
which is the universally adopted convention in Lean.

Observe that the implementation of second order
arithmetic in a given logical system can be nuanced, so it is not a choice
that can be sensibly made in an arbitrary manner. For instance,
inside Lean itself, it is not a given that the finite
von Neumann ordinals in a model of~$\ZFC$ can be seen as a model of second
order arithmetic. In fact, Lean can construct a model of~$\ZFC+\neg\text{Con}(\ZFC)$.
This is, {\it a fortiori}, a model of~$\ZFC$, yet its finite von Neumann ordinals satisfy
$\neg\text{Con}(\ZFC)$, as opposed to $\Nat$.

\noindent With the notation established above, our main result is as
follows:

\setlength{\leftskip}{\parindent}
\setlength{\rightskip}{\parindent}
{\it \noindent Let $\varphi$ be a sentence in the language of second order
arithmetic. Then $\ZFC_\infty\vdash\varphi_Z$ if and only if $\CIC^+$ proves that the
type $\varphi_C$ is inhabited {\rm(}\ie $\CIC^+\vdash e:\varphi_C$ for
some expression $e${\rm)}.}

\setlength{\leftskip}{0em}
\setlength{\rightskip}{0em}
\noindent The proof shows the following chains of implications
$$\begin{aligned}
\ZFC_n &\vdash \varphi_Z
&&\;\;\Longrightarrow\;\;
&\CIC^+ &\vdash\left(\mathcal M_{\ZFC_n}\models\varphi_Z\right)
&&\;\;\Longrightarrow\;\;
&\CIC^+ &\vdash\varphi_C\\
\CIC^+ &\vdash\varphi_C
&&\;\;\Longrightarrow\;\;
&\ZFC_\infty &\vdash\left(\mathcal N_{\CIC^+}\models\varphi_C\right)
&&\;\;\Longrightarrow\;\;
&\ZFC_\infty &\vdash\varphi_Z
\end{aligned}$$
where $\mathcal M_{\ZFC_n}$ is Werner's model of~$\ZFC$ with $n$
inaccessibles in~$\CIC$, and $\mathcal N_{\CIC^+}$ is Carneiro's set
theoretic model of~$\CIC^+$.  The implication $\CIC^+\vdash\left(\mathcal
M_{\ZFC_n}\models\varphi_Z\right) \Rightarrow \CIC^+\vdash\varphi_C$ has
been implemented in Lean~4.

Our proof entails the feasibility of automatic translations from Lean~4
proofs of second order arithmetic statements to, for instance, Isabelle/ZF
(and vice-versa). This is a nontrivial endeavor, as it requires to
formalize Carneiro's work in Isabelle/ZF, thus obtaining a Lean~4 kernel
which is formally verified in Isabelle/ZF. A similar work has been carried
out by Krauss and Schropp~\cite{KrSc10}, providing an automatic translator
from Isabelle/HOL to Isabelle/ZF.  If at all practical, the benefit of
such a translator would be the ability to translate complex Lean~4 proofs
to independently verifiable set theoretic ones.

\section{Definition of $\PAt$ and Translations}

We will follow Carneiro's formalization of Lean, see~\cite{Car19}, as
essentially the type theory of $\CIC$ plus some additional axioms.
To summarize, the grammar of $\CIC^+$ is as follows:
\begin{align*}
\ell	&= u\:|\:0\:|\:S\ell\:|\:\max(\ell,\ell)\:|\:\text{imax}(\ell,\ell)&\text{(universe level)}\\
\Gamma	&= ()\:|\:\Gamma,x:e&\text{(context)}\\
e	&= x\:|\:U_\ell\:|\:e\:\:e\:|\:\lambda x:e\mapsto e\:|\:\forall x:e,e\:|\\
	&\quad\;\mu x:F,K\:|\:\texttt{c}_{\mu x:F,K}\:|\:\text{rec}_{\mu x:F,K}&\text{(expression)}\\
K	&= 0\:|\:(\texttt{c}:e)+K&\\
F	&= U_\ell\:|\:\forall x:e,F&
\end{align*}
$U_\ell$ are the universes; in particular $U_0$ is the sort of propositions.
The usual definitions apply for $\alpha\to\beta$ and the inductive types
$a = b$, $\alpha\times\beta$, and $\exists a:A,P\:a$.
$\CIC^+$ is endowed with $(\beta),(\eta)$ reduction rules,
as well as rules for recursors, and proof
irrelevance. To this, we add the axioms of propositional extensionality
and choice.
We denote by $\CIC^+_n$ (for $n$ natural) the fragment of $\CIC$ that only uses universes up to $U_n$.

Recall that the type $\Nat:U_1$ has an inductive definition given by
$$\Nat := \mu t:U_1,(\Zero:\:\:t)+(\texttt{Succ}:\:\:t\to t)$$
and sum and product are defined via the recursor.

We define the language of second order arithmetic following \cite{Sim09}:
\begin{definition}A {\it second order arithmetic formula} is a formula of two-sorted first-order logic (where the sorts are denoted $s_1$, called the numbers, and $s_2$, called the subsets) in the language $\mathcal L_{\PAt}=\{0,S,+,\times,\vepsilon\}$ (where $0:s_1$, $S:s_1\to s_1$, $+,\times:s_1\times s_1\to s_1$ and $\vepsilon\subset s_2\times s_1$).
(the intended meaning of $\vepsilon(P,t)$ is that $t$ is an element of the set $P$, which we'll sometimes also write as $P(t)$ for brevity.)

An {\it arithmetic sentence} is a second order arithmetic formula with no free variables. The set of arithmetic formulas is denoted $\mathcal F_{\PAt}$.

We will usually indicate with $x_1,x_2,\dots$ the number variables and with $P_1,P_2,\dots$ the subset variables. We will use the symbols $\forall,\exists$ to quantify over numbers and $\forall^2,\exists^2$ to quantify over subsets. Without loss of generality, we may restrict the logical connectives to the complete set~$\{\to,\bot\}$.
\end{definition}
To illustrate the expressivity of second order arithmetic: the standard model, where
the numbers are in~$\mathbb N$ and the subsets range over~$\wp(\mathbb N)$, is bi-interpretable
with the first order structure~$(\mathbb R,+,\cdot)$ expanded by a predicate for~$\mathbb Z\subset\mathbb R$.
The sets that this structure can describe are precisely the projective sets of
descriptive set theory~\cite[exercise 37.6]{kechris1995classical}. Most subsets of~$\mathbb R^n$ used
in everyday mathematics, such as the graphs of continuous functions, are Borel sets, thus also projective sets.
One limitation of this logic is that, although it can construct specific real functions, it does not
allow quantification over the set of all functions~$\mathbb R \to \mathbb R$.

Now, for any formula $\varphi\in\mathcal F_{\PAt}$, we are going to define
its set-theoretic interpretation $\varphi_Z\in\mathcal F_{\ZFC_\infty}$ as a formula in the language $\{\in\}$; and similarly, the type-theoretic
interpretation of $\varphi\in\mathcal F_{\PAt}$ will be an expression $\varphi_C$ of~$\CIC^+$. 
The type of $\varphi_C$ will be the type $\contexts\to\Prop$ of functions
from an appropriate type of contexts to propositions. When $\varphi$ is a sentence, $\varphi_C$ will turn out to be a constant function,
which can therefore be identified with a propositional type representing the sentence~$\varphi$.

\begin{definition}[$\PAt\rightsquigarrow\ZFC_\infty$] 
Given an arithmetic formula~$\varphi\in\mathcal F_\PAt$ we let~$\varphi_Z$ be the customary
translation of~$\varphi$ through the set-theoretical interpretation of second
order arithmetic having sorts~$\omega$ (the set of finite Von Neumann
ordinals) and~$\wp(\omega)$.
\end{definition}
The formal definition of this translation is
straightforward yet rather long: it can be obtained by the implementation of~{\tt AriFormula2\_to\_ZFC'} in the Lean code~\cite{github-repo}.

\begin{definition}[$\PAt\rightsquigarrow\CIC^+$] 
The type $\contexts$ of contexts is~$(\Nat \to \Nat)\times(\Nat \to \Set\Nat)$.
Given an arithmetic formula $\varphi\in\mathcal F_\PAt$, we define
the function~$\varphi_C:\mathcal \contexts\to\Prop$
by structural recursion on the formula~$\varphi$. The sorts are
$\Nat$ and~$\Set\Nat$. Terms are mapped sending the symbols for $0$, $+$
and~$\times$ to the corresponding operators defined pointwise on
$\contexts\to\Nat$, and the variables to the projections 
$(x_i)_C\:\langle\mathcal C_1,\mathcal C_2\rangle\equiv\mathcal C_1\:i$,
$(P_j)_C\:\langle\mathcal C_1,\mathcal C_2\rangle\equiv\mathcal C_2\:j$.
Syntactic elements are mapped to the natural type theoretic
constructions as follows:
\begin{itemize}
\item $\bot_C\:\:\langle\mathcal C_1,\mathcal C_2\rangle
\equiv \bot$
\item $(\varphi\to\delta)_C\:\:\langle\mathcal C_1,\mathcal C_2\rangle
\equiv (\varphi_C\:\langle\mathcal C_1,\mathcal C_2\rangle) \to
(\delta_C\:\langle\mathcal C_1,\mathcal C_2\rangle)$;
\item $(\exists x_i,\varphi)_C\:\:\langle\mathcal C_1,\mathcal
C_2\rangle\equiv\exists n:\Nat,\:\varphi_C\:\langle\mathcal C_1[n/i],\mathcal C_2\rangle$;
\item $(\forall x_i,\varphi)_C\:\:\langle\mathcal C_1,\mathcal
C_2\rangle\equiv\forall n:\Nat,\:\varphi_C\:\langle\mathcal C_1[n/i],\mathcal C_2\rangle$;
\item $(\exists^2 P_i,\varphi)_C\:\:\langle\mathcal C_1,\mathcal
C_2\rangle\equiv\exists Q:\Set\Nat,\:\varphi_C\:\langle\mathcal C_1,\mathcal C_2[Q/i]\rangle$;
\item $(\forall^2 P_i,\varphi)_C\:\:\langle\mathcal C_1,\mathcal
C_2\rangle\equiv\forall Q:\Set\Nat,\:\varphi_C\:\langle\mathcal C_1,\mathcal C_2[Q/i]\rangle$;
\end{itemize}
where $\mathcal C_1[n/i]$ is the function that agrees with~$\mathcal C_1$ except when the
argument is~$i$, and gives~$n$ when the argument is~$i$, and similarly for~$\mathcal C_2[Q/i]$.

When $\varphi$ is a sentence, it easy to see that~$\varphi_C$ is necessarily a
constant function. By abuse of notation we will denote value assumed by
this function by the same symbol~$\varphi_C$.
\end{definition}

\begin{remark} Notably, when $\varphi$ is a sentence, a syntactic
representation of the propositional type $\varphi_C$
can be obtained by little more than just copying the formula.
Indeed it is exactly because the
translation is so trivial that we dare say that Lean~4, or~$\CIC^+$, is
designed to encourage its user to consider $\varphi_C$ and $\varphi$,
intuitively, the {\it same} formula.
\end{remark}
Our main result is the following:
\begin{theorem}\label{MAIN}Let $\varphi$ be an arithmetic sentence (\ie a
formula with no free variables). Then $\ZFC_\infty$ proves $\varphi_Z$ if
and only if $\varphi_C$ is inhabited in $\CIC^+$.
\end{theorem}
In the next sections we will outline the proof of each direction of
the equivalence.

\section{Every Arithmetic Formula Provable in $\ZFC_\infty$ Is Provable in $\CIC^+$}
\begin{teoa}\label{z2l}Let $\varphi$ be an arithmetic sentence, and assume that $\ZFC_\infty$ proves $\varphi_Z$. Then $\varphi_C$ is inhabited in $\CIC^+$.
\end{teoa}
By~\cite{Wer97}, $\ZFC_n$ can be coded into $\CIC_{n+2}^+$ as follows:
$$\SET_n := \mu t:U_{n+1},(\texttt{sup}:\forall A:U_n,(A\to t)\to t)$$
$$\begin{aligned}\EXTn (\supt A\:&\:f)\:(\supt B\:\:g):=\\
	&\left(\forall a:A,\exists b:B,\EXTn (f\:\:a)\:(g\:\:b)\right)\:\:\land\\
	&\left(\forall b:B,\exists a:A,\EXTn (f\:\:a)\:(g\:\:b)\right)\end{aligned}$$
$$\INn x\:\:(\supt A\:\:f):=\exists a:A,\EXTn x\:\:(f\:\:a).$$
\begin{fact}[Werner, 1997 \cite{Wer97}]\label{Wer}Let $1\le n\in\N$.
For each closed formula $\psi$ in the language $\{\in\}$ of set theory,
construct a propositional type $\psi_\mathcal M$ by interpreting the
equality with $\texttt{EXT}_n$, the inclusion with $\texttt{IN}_n$ and by
letting the quantifiers range over $\SET_n$. Then, whenever $\psi$ is an
axiom of $\ZFC_{n-1}$, one can construct in $\CIC_{n+1}^+$ an inhabitant
of $\psi_\mathcal M$.
\end{fact}
Note that it is an immediate consequence of Fact~\ref{Wer} that $\ZFC_\infty\models\varphi_Z$ implies that, for $n$ large enough, $(\varphi_Z)_\mathcal M$ is inhabited, where $\mathcal M$ is the model of $\ZFC_n$ in $\CIC^+_{n+2}$ given by Fact~\ref{Wer}. As such, to show Theorem~\ref{MAIN}($\Rightarrow$) it is sufficient to exhibit an inhabitant of the type $(\varphi_Z)_\mathcal M\to\varphi_C$.

We indeed show slightly more:
\begin{lemma}\label{Z2Lsimpcompl}Let $n\ge1$ be arbitrary, and $\SET:=\SET_n$. There are functions $\texttt{num:}\Nat\to\SET\to U_0$ and $\texttt{sub:}\Set\Nat\to\SET\to U_0$ so that the following holds. For any arithmetic formula $\varphi$, let $x_i$ denote the free number variables of $\varphi$ and $P_j$ the free subset variables of $\varphi$ (for $i,j$ ranging over some finite subsets of the integers), and let $(n_i:\Nat)$, $(Q_j:\Set\Nat)$, $(y_i:\SET)$, $(z_j:\SET)$ be s.t.
\begin{enumerate}[(i)]
\item $\texttt{num }n_i\:\:y_i$ is inhabited for each $i$
\item $\texttt{sub }Q_j\:\:z_j$ is inhabited for each $j$.
\end{enumerate}
Then the following propositional type is inhabited:
$$\varphi_C[n_i/x_i\dots,Q_j/P_j\dots]\leftrightarrow(\varphi_{Z})_\mathcal M[y_i/x_i\dots,z_j/P_j\dots].$$
\end{lemma}
We provide a Lean~4 formalization of this lemma and of \teoaname~\cite{github-repo};
references to the Lean code will be provided in footnotes.

The proof\footnote{In the code, {\tt ZFC\_iff\_Lean'}.} is carried out by structural recursion on $\varphi$; the general sketch is as follows:
\begin{enumerate}[(i)]
\item Construct $\texttt{num}$ by structural recursion on its first argument, so that intuitively $\texttt{num}\:\:n\:\:x$ is true when ``$x$ is the $n$-th finite ordinal'' is true in $\mathcal M$; and similarly produce $\texttt{sub}$ s.t. $\texttt{sub}\:\:P\:\:x$ is true when $(x\subset\omega)_\mathcal M$ and, for any $n$, ``$x$ contains the $n$-th finite ordinal''$_\mathcal M$ iff $P\:\:n$. Then, one verifies that $\texttt{num}$, $\texttt{sub}$ are the graphs of bijective maps which are well-behaved w.r.t. the operations of arithmetic, hence that $\Nat$ is isomorphic to $(x:\SET)\times(x\in\omega)_\mathcal M$, and that $\Set\Nat$ is isomorphic to $(x:\SET)\times(x\subset\omega)_\mathcal M$.\footnote{In the code, those two relations are provided by {\tt is\_number} and {\tt is\_real} resp. The fact that those maps are bijections that preserve the structure is provided by {\tt isoN} and the theorems from {\tt plus\_is\_sum} to {\tt succ\_is\_succ}.}
\item Construct an inductive type $\texttt{Form}$ of arithmetic formulas (with variables indexed with $\Nat$), and associate a $\bar\varphi:\texttt{Form}$ to each formula $\varphi$.\footnote{In the code, this inductive type is {\tt AriFormula2}, which uses the inductive types {\tt AriTerm} and {\tt AriTerm2} to represent terms of first and second sort resp.}
\item Construct valuations $\texttt{val}_C:(\Nat\to\Nat)\to(\Nat\to\Set\Nat)\to\texttt{Form}\to\Prop$ that, given an $f:\Nat\to\Nat$ valuation of the number variables, $g:\Nat\to\Set\Nat$ valuation of the subset variables and $\bar\varphi:\texttt{Form}$, outputs a true proposition iff $\varphi_C[(f\:\:i)/x_i\dots,\:(g\:\:j)/P_j\dots]$ is true; and similarly $\texttt{val}_Z:(f:\Nat\to\SET)\to(g:\Nat\to\SET)\to(\bar\varphi:\texttt{Form})$ that outputs a true proposition iff $(\varphi_Z)_\mathcal M[(f\:\:i)/x_i\dots,\:(g\:\:j)/P_j\dots]$ is true.\footnote{In the code, those maps are {\tt AriFormula2\_to\_Lean'} and {\tt AriFormula2\_to\_ZFC'} resp.}
\item Finally check that $(\texttt{val}_C\:\:f\:\:g\:\:\bar\varphi)\leftrightarrow(\texttt{val}_Z\:\:f\pr \:\:g\pr\:\:\bar\varphi)$ is inhabited as long as $\texttt{num}\:\:(f\:\:i)\:\:(f\pr\:\:i)$ whenever $x_i\in FV(\varphi)$ and $\texttt{sub}\:\:(g\:\:j)\:\:(g\pr\:\:j)$ whenever $P_j\in FV(\varphi)$, by structural recursion on $\bar\varphi$.
\end{enumerate}

In the end, we will have constructed an inhabitant $\texttt{C\_iff\_Z}$ of the propositional type
$$\begin{aligned}
\forall \bar\varphi:\texttt{Form},&\:\forall f:\Nat\to\Nat,\:\forall g:\Nat\to\Set\Nat,\\
&\:\forall f\pr:\Nat\to\SET,\:\forall g\pr:\Nat\to\SET,\\
&\quad(h:\:\forall n:\Nat,\text{``$x_n$ is free in $\varphi$''}\:\to \texttt{num}\:\:(f\:\:n)\:\:(f\pr\:\:n)) \to\\
&\quad(h\pr:\:\forall n:\Nat,\text{``$P_n$ is free in $\varphi$''}\:\to \texttt{sub}\:\:(g\:\:n)\:\:(g\pr\:\:n)) \to\\
&\quad\quad\big((\texttt{val}_C\:\:f\:\:g\:\:\bar\varphi)\leftrightarrow(\texttt{val}_Z\:\:f\pr
\:\:g\pr\:\:\bar\varphi)\big)
\end{aligned}$$

\begin{proof}[\teoaname]\footnote{In the code, this is provided by {\tt ZFC\_iff\_Lean}.} Let $\varphi$ be an arithmetic sentence, and
assume $\ZFC_\infty\models\varphi_Z$. From Fact~\ref{Wer}, then
$(\varphi_Z)_\mathcal M$ is inhabited, let $\alpha:(\varphi_Z)_\mathcal
M$. Applying $\texttt{C\_iff\_Z}\;\bar\varphi$ to arbitrary contexts we
obtain an inhabitant of $\varphi_C \leftrightarrow (\varphi_Z)_\mathcal M$. We
conclude,
applying the $\leftarrow$ arrow of this to $\alpha$. \qed
\end{proof}
\section{Every Arithmetic Formula Provable by Lean is Provable in $\ZFC_\infty$}
\begin{teob}{}\label{l2z}Let $\varphi$ be an arithmetic formula with no free variables, and assume that $\varphi_C$ is inhabited in $\CIC^+$. Then, $\ZFC_\infty$ proves $\varphi_Z$.
\end{teob}
To prove that, we will follow the following scheme:
\begin{enumerate}[(i)]
\item Following Carneiro \cite{Car19}, for each $\varphi$ one can construct, assuming the existence of $n$ strongly inaccessible cardinals for some sufficiently large $n$ (which may, of course, depend on $\varphi$) a set $\lsq\varphi\rsq$ which $\ZFC_n$ proves to be nonempty precisely if $\varphi_C$ is inhabited.
\item We then need to show that \[\ZFC_n\vdash(\lsq\varphi\rsq\not=\varnothing)\longrightarrow\varphi_Z.\]
From that, we can conclude that $\ZFC_n\vdash\varphi_Z$ whenever
$\varphi_C$ is inhabited.
\end{enumerate}

\noindent Carneiro describes, in~\cite[section~5]{Car19}, a model of the fragment $\CIC^+_W$ of
$\CIC^+$ that uses only the $W$-type as a general inductive type,
plus a handful of simple constructors.
$W$-types are those produced by the $W$ constructor, as defined in the
following definition:
\begin{definition} (Martin-L\"{o}f 1982, \cite{Lof82}) Let $\Gamma\vdash
a:U_\ell$ and $\Gamma,\:a:U_\ell\vdash b:U_\ell\pr$; then $$Wx:a,b\::=\mu
t:U_{\max(\ell,\ell\pr,1)},\:(\texttt{sup: } a\to(b\to t)\to
t).$$
\end{definition}
Then, \cite{Car19} shows a reduction of the type theory of Lean onto this fragment.
Technically, this means that we need to work with the model of~$\CIC^+_W$
directly.

Let us begin with the definition of $\Natw$ obtained as the special case
of the translation described in \cite[section 5.2]{Car19} applied to the inductive type $\Nat$.
\[\begin{aligned}\Natw:U_1\::=Wx:\top\oplus\top,\texttt{rec}_\oplus\:&(\lambda
\_:\top\:\mapsto \bot)\:(\lambda \_:\top\:\mapsto \top)\:x.\end{aligned}\]
One can compute the $\CIC_W^+$ terms $0:\Natw$, $S:\Natw\to\Natw$,
$+:\Natw\to\Natw\to\Natw$, $\times:\Natw\to\Natw\to\Natw$ which
correspond to the usual operators on $\Nat$ via the reduction to
$\CIC^+_W$.
To any arithmetic formula $\varphi$ we associate $\varphi_W:\contexts\to U_0$ in the in the same manner as $\varphi_C$, except by quantifying over $\Natw$ instead of $\Nat$.

What we gain from those contortions is that now $\varphi_W$ is a type of
the fragment $\CIC^+_W$; so we would like to use it in place of
$\varphi_C$. We can do this as long as $\varphi_C\leftrightarrow\varphi_W$ is inhabited; for then $\varphi_W$ would be inhabited iff $\varphi_C$ is.

The following lemma is proved by a straightforward structural induction:
\begin{lemma}\label{L2W}There are bijections $a:\Nat\to\Natw$ and $A:(\Set\Nat)\to(\Set\Natw)$ s.t. the following holds. Whenever $\varphi$ is an arithmetic formula with free number variables $x_i$ and free subset variables $P_j$, and $n_i:\:\Nat\dots,\:Q_j:\:\Set\Nat\dots$, then $\varphi_C[n_i/x_i\dots,Q_j/P_j\dots]\leftrightarrow\varphi_W[(a\:n_i)/x_i\dots,(A\:Q_j)/P_j\dots]$ is inhabited.

In particular, if $\varphi$ is an arithmetic sentence, then the propositional type $\varphi_C\leftrightarrow\varphi_W$ is inhabited.
\end{lemma}
Given a (strictly) increasing sequence of cardinals
$(\gamma_n)_{n\in\omega}$, we say that it is {\it $k$-correct} if $\gamma_i$ is inaccessible for all $0\le i<k$.
\begin{fact}[Carneiro \cite{Car19}, 2019]\label{Car} Let $(\kappa_n)_{n\in\omega}$ be an increasing succession of strong limit cardinals, and let
$U_0=\{\varnothing,\{\bullet\}\}$, $U_{n+1}=V_{\kappa_{n+1}}$, $U_\omega=\bigcup_nU_n$ (where $\bullet=\varnothing$ is the proof object). Then there exist maps (where $|\Gamma|$ denotes the length of a given context $\Gamma$)
$$\begin{aligned}
	\Gamma\text{ well-formed context }
		\quad&\mapsto\quad
	\lsq\Gamma\rsq\subset U_\omega^{|\Gamma|}\\
	\begin{matrix}
		\big(\text{$\Gamma$ well-formed context, $\gamma\in\lsq\Gamma\rsq$},\\
		\quad\quad\quad\quad\quad\quad\text{$e$ s.t. $\Gamma\vdash e:\alpha$ for some $\alpha$}\big)
	\end{matrix}
		\quad&\mapsto\quad
	\lsq\Gamma\vdash e\rsq_\gamma\in U_\omega
\end{aligned}$$
(where $e$ ranges over the expressions of $\CIC^+_W$) so that:\begin{enumerate}[i.]
\item if $\Gamma\vdash\alpha:U_0$, then $\lsq\Gamma\vdash\alpha\rsq_\gamma\subseteq\{\bullet\}$
\item if $\Gamma\vdash e:\alpha$, and if $\kappa_n$ is $k$-correct for a sufficiently large $k$, then for all $\gamma\in\lsq\Gamma\rsq$, $\lsq\Gamma\vdash e\rsq_\gamma\in\lsq\Gamma\vdash\alpha\rsq_\gamma$
\item if $\Gamma\vdash e\equiv e\pr$, and if $\kappa_n$ is $k$-correct for a sufficiently large $k$, then for all $\gamma\in\lsq\Gamma\rsq$, $\lsq\Gamma\vdash e\rsq_\gamma=\lsq\Gamma\vdash e\pr\rsq_\gamma.$
\end{enumerate}
\end{fact}
In the proof of Fact~\ref{Car} in \cite{Car19}, $\lsq\Gamma\rsq,\lsq\Gamma\vdash e\rsq_\gamma$ are defined by mutual recursion. The intended meaning is that $\lsq\Gamma\vdash e\rsq_\gamma$ is the set that represents $e$ whenever the variables that appear free in it (which are typed in the context $\Gamma$) are interpreted in accordance with $\gamma$; and $\lsq\Gamma\rsq$ is the set of vectors of the possible set-theoretical interpretations of the variables it types. 

In the rest of this section, we will use the particular definition of $\lsq\cdot\rsq$ given in the proof of Fact~\ref{Car}. 

In particular, by explicitly doing the calculations, one finds that
\begin{lemma}\label{Nok}There is a countable set $\tilde\omega$ s.t., for any well-formed context $\Gamma$ and $\gamma\in\lsq\Gamma\rsq$, $\lsq\Gamma\vdash\Natw\rsq_\gamma=\tilde\omega$. Furthermore, there is a bijection $\omega\to\tilde\omega$, $n\mapsto\overline n$ s.t.:\begin{itemize}
\item $\lsq\Gamma\vdash 0\rsq_\gamma=\overline 0$.
\item $\tilde S=\lsq\Gamma\vdash S\rsq_\gamma$ is the map $\tilde\omega\to\tilde\omega$ given by $\tilde S(\overline n)=\overline{n+1}$.
\item $\tilde +=\lsq\Gamma\vdash +\rsq_\gamma$ is the map $\tilde\omega\to\tilde\omega^{\tilde\omega}$ given by $\tilde+(\overline n)(\overline m)=\overline{n+m}$.
\item $\tilde \times=\lsq\Gamma\vdash \times\rsq_\gamma$ is the map $\tilde\omega\to\tilde\omega^{\tilde\omega}$ given by $\tilde\times(\overline n)(\overline m)=\overline{nm}$.
\end{itemize}
\end{lemma}
We can now undertake the step (ii) of the proof:
\begin{lemma}\label{W2Z}Let $\varphi$ be an arithmetic formula with free number variables in $(x_i)_{i=1\dots p}$ and free subset variables in $(P_j)_{j=1\dots q}$; let moreover $n_1,\dots,n_p\in\omega$, $A_1,\dots,A_q\in\wp(\omega)$; then, if $\kappa_n$ is $k$-correct for a sufficiently large $k$,
$$\bullet\in\lsq x_i:\Natw\dots,P_j:\Set\Natw{}\dots{} \vdash \varphi_W\rsq_{\overline{n_i}\dots,\overline{A_j}\dots}\;\; \Longleftrightarrow\;\;\varphi_Z[n_i/x_i\dots,A_j/P_j\dots]$$
\end{lemma}
The proof is by structural recursion. The base step follows from showing that, for any term $t$, one has $\lsq x_i:\Natw\dots{}\vdash t_W\rsq_{\overline{n_i}\dots}=\overline{t_Z[n_i/x_i\dots]}$ which follows using the fact that $\tomega$ is ``well-behaved'' with respect to the arithmetic operations.

\begin{proof}[\teobname]
Let $\varphi$ be an arithmetic formula with no free variables, and assume that $\varphi_C$ is inhabited by $e:\varphi_C$. Now, {\it working in $\ZFC_\infty$}, we set out to prove $\varphi_Z$.

Using the Lemma~\ref{L2W}, we can pick an $\vdash e\pr:\varphi_W$. In order to use the map $\lsq\cdot\rsq$, let $v$ be an universe valuation that sets all universe variables to 0, and $\vdash\langle e\pr\rangle_v:\varphi_W$ (see the theorems 6.2 and 6.3 in \cite{Car19}). Now assume that $\kappa_n$ is $k$-correct for a sufficiently large $k$ as to satisfy the hypothesis of Fact~\ref{Car}. Let $p=\lsq\vdash \langle e\pr\rangle_v\rsq_{()}$, and by Fact~\ref{Car},(ii) one has that $p\in\lsq\vdash\varphi_W\rsq_{()}$; but, by (i), $\lsq\vdash\varphi_W\rsq_{()}\subseteq\{\bullet\}$, hence $p=\bullet$, and by the Lemma~\ref{W2Z} then $\varphi_Z$ is true.

We have now proven that $\ZFC_\infty$ shows $\varphi_Z$ if $\varphi_C$ is inhabited.\qed
\end{proof}
\bibliography{biblio}
\end{document}